\documentclass[12pt, twoside, leqno]{article}

\usepackage{amsmath,amsthm,mathrsfs}
\usepackage{amssymb}

\usepackage{enumitem}

\usepackage{graphicx}

\usepackage[T1]{fontenc}

\newtheorem{theorem}{Theorem}[section]

\theoremstyle{definition}

\numberwithin{equation}{section}

\makeatletter
\newcommand\tint{\mathop{\mathpalette\tb@int{t}}\!\int}
\newcommand\bint{\mathop{\mathpalette\tb@int{b}}\!\int}
\newcommand\tb@int[2]{%
  \sbox\z@{$\m@th#1\int$}%
  \if#2t%
    \rlap{\hbox to\wd\z@{%
      \hfil
      \vrule width .35em height \dimexpr\ht\z@+1.4pt\relax depth -\dimexpr\ht\z@+1pt\relax
      \kern.05em 
    }}
  \else
    \rlap{\hbox to\wd\z@{%
      \vrule width .35em height -\dimexpr\dp\z@+1pt\relax depth \dimexpr\dp\z@+1.4pt\relax
      \hfil
    }}
  \fi
}
\makeatother

\def\R{\mathbb{R}}

\def\N{\mathbb{N}}

\def\F{\mathbb{F}}
\def\P{\mathbb{P}}

\def\Q{\mathbb{Q}}

\def\lbrak{\left[}
\def\rbrak{\right]}
\def\abClosed{\lbrak a,b\rbrak}
\def\lpar{\left(}
\def\rpar{\right)}
\def\abOpen{\lpar a,b\rpar}

\def\into{\longrightarrow}
\def\implies{\Longrightarrow}

\def\setdiff{\backslash}

\def\CA{{\bf(CA)}}

\def\LIP{{\bf(LIP)}}
\def\LMP{{\bf(LMP)}}

\def\LPi{{\bf(LP1)}}
\def\LPii{{\bf(LP2)}}
\def\LPiii{{\bf(LP3)}}

\def\mCA{\mbox{\bf(CA)}}

\def\mLIP{\mbox{\bf(LIP)}}
\def\mLMP{\mbox{\bf(LMP)}}

\def\mLPi{\mbox{\bf(LP1)}}
\def\mLPii{\mbox{\bf(LP2)}}
\def\mLPiii{\mbox{\bf(LP3)}}

\def\theSet{\mathcal{S}}

\def\sub{\subseteq}

\def\nseqLim{\displaystyle\lim_{n\rightarrow\infty}}

\def\aseq{\lpar a_n\rpar}

\def\bseq{\lpar b_n\rpar}

\def\Int{\displaystyle\int}

\def\Star{{\bf($\star$)}}

\def\symdiff{\triangle}

\def\Star{{\bf($\star$)}}

\def\mStar{\mbox{\bf($\star$)}}

\begin{document}


\baselineskip=17pt


\title{Littlewood's principles in reverse real analysis}

\author{Rafael Reno S. Cantuba\\
Department of Mathematics and Statistics\\ 
De La Salle University\\
2401 Taft Ave., Manila\\
1004 Metro Manila, Philippines\\
E-mail: rafael\_cantuba@dlsu.edu.ph}

\date{}

\maketitle


\renewcommand{\thefootnote}{}



\renewcommand{\thefootnote}{\arabic{footnote}}
\setcounter{footnote}{0}


\begin{abstract}
If local forms of Littlewood's three principles are stated as  axioms for an ordered field, then each principle is equivalent to the completeness axiom.
\end{abstract}

\section{Introduction}

In the past decade, the papers \cite{dev14,pro13,tei13} gave an engaging account of how some of the traditional theorems from calculus and real analysis are each equivalent to the least upper bound property of, or completeness axiom for, the field $\R$ of all real numbers. In a mathematical theory, theorems are proven from a collection of axioms, but in the aforementioned papers, the opposite thought process was exhibited: in the appropriate universe of discourse (which in this case is the class of all ordered fields) selected traditional theorems were each stated as an axiom (for an arbitrary ordered field) and the statement of the completeness axiom was proven as a consequence. As the author, J. Propp, of \cite{pro13} pointed out, this intellectual exercise has the flavor of \emph{Reverse Mathematics} \cite[p. 392]{pro13}, which is interesting in its own right. See, for instance, \cite[Section 1.1]{sim09} or \cite[Chapter 1]{sti18}. Propp further pointed out that this kind of investigation ``sheds light on the landscape of mathematical theories and structures,'' and that ``arguably the oldest form of mathematics in reverse'' is the quest for a list of equivalences for the parallel postulate in Euclidean geometry \cite[pp. 392--393]{pro13}. An independent work \cite{tei13} gave more technical proofs and initiated a list of `completeness properties' for an ordered field that was expanded in \cite{dev14}, so that, currently, 72 characterizations of the completeness axiom have been identified. An almost full cast of the traditional calculus theorems, ranging from the Intermediate Value Theorem to convergence tests and L'H\^opital's rule, appear in the list, but we also find some forms of the Arzel\`a-Ascoli Theorem, and the Lipschitz property of $C^1$ functions. This paper is inspired by the question, of possibly including in the list, some of the theorems of measure theory. We decided to start with the most fundamental\textemdash Littlewood's three principles, the well-known heuristics for understanding measure theory. 

A quote from \emph{the} J. E. Littlewood is now inevitable. From \cite[p. 26]{lit44}, the three principles are: 
\begin{enumerate}\item\label{prin1} Every measurable set is nearly a finite sum (meaning union) of intervals.
\item\label{prin2} Every (measurable) function is nearly continuous.
\item\label{prin3} Every convergent sequence of (measurable) functions is nearly uniformly convergent.
\end{enumerate}
The principle \ref{prin2} is also known as Lusin's Theorem \cite[pp. 72, 74]{roy88}, while \ref{prin3} is also referred to as Egoroff's Theorem \cite[p. 40]{lit44}. We shall keep the form of each of the above principles as a `one-directional' implication. For instance, \ref{prin2} may be rephrased as ``If $f$ is a measurable function, then $f$ is nearly continuous.'' In the textbooks, we often find the biconditional form ``$f$ is a measurable function if and only if $f$ is nearly continuous,'' but this is \emph{not} what we shall use. The choice of which statement appears in the hypothesis and which statement appears in the conclusion of the one-directional implication is based on how J. E. Littlewood originally stated the principles: the `nearly' part is always in the conclusion of the conditional statement. Principles \ref{prin1} and \ref{prin3} shall be handled similarly. Also, we shall be using `local' forms of such principles. That is, there is a given closed and bounded interval $I$ (in an arbitrary ordered field) such that we shall consider only the measurable sets contained in $I$ and functions with domain $I$ (or a subset of $I$). This `local' perspective means that, among the many forms of the first principle, we shall indeed be using that form which involves the symmetric difference of two sets, where the second set is the union of a finite number of intervals, which was originally in the aforementioned quote from J. E. Littlewood.

The issue of how to define Lebesgue measure and Lebesgue integrals in an arbitrary ordered field $\F$ was one of the first issues we had to deal with. If the usual outer measure approach is to be used, then we have to define outer measure as the infimum of some set, but then, the \emph{Existence of Infima}, or the assertion that any nonempty subset of $\F$ that has a lower bound has an infimum, is one of the equivalent forms of the completeness axiom \cite[p. 108]{tei13}, yet $\F$ need not be complete. What we found as a suitable approach is the Riesz method \cite[Chapter II]{cha95}, in which step functions form the starting point for establishing the definition of the Lebesgue integral, and the measurability of a set is defined later as the Lebesgue measurability of its characteristic function.

\section{Preliminaries}\label{PrelSec}

Let $\F$ be an ordered field. Thus, $\F$ has characteristic zero, and consequently, contains a subfield isomorphic to the field $\Q$ of all rational numbers. The usual ordering in $\Q$ is consistent with the order relation $<$ on $\F$ defined by $a<b$ if and only if $b-a\in\P$, where $\P$ is the set of all positive elements of $\F$.  The relation $<$ on $\F$ obeys a \emph{Trichotomy Law} which states that for any $a,b\in\F$, exactly one of the assertions $a<b$, $b<a$ or $a=b$, is true.

The relation $>$ is defined by $a>b$ if and only if $b<a$, and the relations $<$ and $>$ may be extended in the usual manner to obtain the relations $\leq$ and $\geq$, respectively. Using these order relations, some standard notions of analysis may be defined for the field $\F$. 

Given $a,b\in\F$, the open interval with left endpoint $a$ and right endpoint $b$ is $\abOpen:=\{x\in\F\  :\  a<x<b\}$. The intervals $\abClosed$, $\lpar a,b\rbrak$ and $\lbrak a,b\rpar$ may be defined in the obvious manner, by taking the union of $\abOpen$ with one or both of its endpoints. The length of any of the four aforementioned intervals is defined to be $b-a$. 

Given functions $f,g:\abClosed\into\F$, by $f\geq g$, we mean $f(x)\geq g(x)$ for all $x\in\abClosed$. A function $f:\abClosed\into\F$ is continuous at $c\in\abClosed$ if, for each $\varepsilon\in\P$, there exists $\delta\in\P$ such that for any $x\in\abClosed\cap\lpar c-\delta,c+\delta\rpar$, $f(x)-f(c)\in\lpar-\varepsilon,\varepsilon\rpar$. We say that $f:\abClosed\into\F$ is a continuous function if $f$ is continuous at each element of $\abClosed$.

Since $\F$ contains $\Q$ as a subfield, we may view the set of all positive integers $\N\sub\Q$ as a subset of $\F$. A sequence in $\theSet\sub\F$ is a function\linebreak $a:\N\into\theSet$. The traditional notation is $a_n:=a(n)$ for any $n\in\N$, and instead of referring to $a$ as a sequence, we say that $\aseq$ is a sequence. If indeed $\aseq$ is a sequence in some subset of $\F$, $a_n$ may further be equal to some other expression determined by $n$. A sequence $\aseq$ converges to $L\in\F$ if, for each $\varepsilon\in\P$, there exists $N\in\N$ such that for any $n\in\N$ with $n\geq N$, we have $a_n-L\in\lpar-\varepsilon,\varepsilon\rpar$. If indeed $\aseq$ converges to $L$, then using the Trichotomy Law in $\F$, the element $L$ is unique, and we define $\nseqLim a_n:=L$. Given sequences $\aseq$ and $\bseq$ such that for each $n\in\N$, $b_n=\displaystyle\sum_{k=1}^na_k$, if $\bseq$ converges, then $\displaystyle\sum_{k=0}^\infty a_k:=\nseqLim b_n$.  

Given $\abClosed\sub\F$, we shall also be considering sequences of functions $\abClosed\into\theSet$ for some $\theSet\sub\F$. By such, we simply mean that each $n\in\N$ is assigned to a unique function $\varphi_n:\abClosed\into\theSet$, and we denote the function sequence by $\lpar\varphi_n\rpar$. We say that $\lpar\varphi_n\rpar$ is \emph{monotonically decreasing} if $\varphi_{n}\geq \varphi_{n+1}$ for any $n\in\N$. Let $\lpar\varphi_n\rpar$ be a sequence of functions $\abClosed\into\theSet$. We say that the functions in $\lpar\varphi_n\rpar$ \emph{converge uniformly} to a function $f:\abClosed\into\theSet$ if for each $\varepsilon\in\P$, there exists $N\in\N$ such that for any $n\geq N$ and any $x\in\abClosed$, $f(x)-\varphi_n(x)\in\lpar-\varepsilon,\varepsilon\rpar$. If indeed the functions in $\lpar\varphi_n\rpar$ converge to $f$ and each $\varphi_n$ is a continuous at any element of $\abClosed$, then by routine instantiation of quantifiers, $f$ is also continuous at any element of $\abClosed$.

By a \emph{partition} of $\abClosed\neq\emptyset$, we mean a finite subset of $\abClosed$ that contains the endpoints $a$ and $b$. If $a=b$, then the only possible partition of $\abClosed$ is the singleton $\{a\}=\{b\}$. For the non-degenerate case, which is when $a<b$, we use the traditional notation, in which, if indeed we have a partition $\Delta$ of $\abClosed$ with cardinality $n$, then the elements of $\Delta$ are indexed as\linebreak $a=x_0<x_1<\cdots<x_n=b$. A function $\varphi:\abClosed\into\F$ is a \emph{step function} if there exists a partition $\{x_0,x_1,\ldots,x_n\}$ of $\abClosed$ such that for each $i\in\{1,2,\ldots,n\}$, the restriction $\left.\varphi\right|_{\lpar x_{i-1},x_i\rpar}$ is a constant function, that is, a function with a one-element range, say $\{M_i\}$ for some $M_i\in\F$. In such a case, the integral of $\varphi$ over $\abClosed$ is defined as 
\begin{eqnarray}
\Int_a^b\varphi:=\sum_{i=1}^nM_i(x_i-x_{i-1}).\label{IntegralStep}
\end{eqnarray}
For the degenerate case $a=b$, the summation in \eqref{IntegralStep} is an empty sum, and so $\Int_a^b\varphi=0$. If $a<b$, and if $M_i\in\P\cup\{0\}$ for any $i\in\{1,2,\ldots,n\}$, then $\Int_a^b\varphi\geq 0$. 

Consider an arbitrary nondegenerate $\abClosed\sub\F$. We say that $\theSet\sub\abClosed$ is a \emph{null set} or \emph{has measure zero} if, for each $\varepsilon\in\P$, there exists a countable collection $\{\lpar a_n,b_n\rpar\sub\abClosed\  :\  n\in\N\}$ of open intervals such that
\begin{eqnarray}
\theSet\sub\bigcup_{n=1}^\infty\lpar a_n,b_n\rpar,\qquad\sum_{n=1}^\infty(b_n-a_n)<\varepsilon.\nonumber
\end{eqnarray}
A statement $\mathscr{P}(x)$ is said to hold \emph{almost everywhere in $\abClosed$} if\linebreak $\{x\in\abClosed\  :\  \neg\mathscr{P}(x)\}$ has measure zero. We say that a sequence $\lpar\varphi_n\rpar$ of functions $\abClosed\into\F$ converges to a function $f:\abClosed\into\F$ almost everywhere in $\abClosed$ if $f(x)=\nseqLim\varphi_n(x)$ almost everywhere in $\abClosed$. A function $f:\abClosed\into\P\cup\{0\}$ is \emph{Lebesgue measurable} if there exists a monotonically decreasing sequence $\lpar\varphi_n\rpar$ of step functions $\abClosed\into\P\cup\{0\}$ that converge to $f$ almost everywhere in $\abClosed$. A function $f:\abClosed\into\P\cup\{0\}$ \emph{has a Lebesgue integral} if $f$ is Lebesgue measurable and if there exists $\Int_a^bf\in\F$ such that for any  sequence $\lpar\varphi_n\rpar$ of monotonically decreasing step functions $\abClosed\into\P\cup\{0\}$ that converge to $f$ almost everywhere in $\abClosed$, the sequence $\lpar\Int_a^b\varphi_n\rpar$ of integrals converges to $\Int_a^bf$. The notion of Lebesgue measurability and of having a Lebesgue integral may then be extended to a function $\abClosed\into\F$ in the usual manner, which is by considering the nonnegative and negative `parts' of a function.

For any sets $X$ and $Y$, we define $X\setdiff Y:=\{x\in X\  :\  x\notin Y\}$,\linebreak and $X\symdiff Y:=(X\setdiff Y)\cup(Y\setdiff X)$.

We say that $E\sub\abClosed$ is a \emph{measurable set} if its \emph{characteristic function} 
\begin{eqnarray}
\chi_E(x):=\begin{cases}1, & \mbox{if }x\in E,\\
0, & \mbox{if }x\in \abClosed\setdiff E,\end{cases}\nonumber
\end{eqnarray}
is a measurable function. If $E$ has measure zero, or $\chi_E$ has a Lebesgue integral, then $E$ \emph{has Lebesgue measure}. In the former case, $\Int_a^b\chi_E=0$.

By a \emph{cut}\footnote{This definition of cut was taken from \cite[p. 60]{mon08}. We chose it because it is apparently more concise. The definition in \cite{dev14,pro13,tei13} is based on the more traditional, which is that a cut is a pair of subsets of $\F$.} of $\F$ we mean a nonempty proper subset $A$ of $\F$ such that for any $a\in A$ and any $b\in\F\setdiff A$, $a<b$. We say that $c\in\F$ is a \emph{cut point} of a cut $A\sub\F$ if for any $a\in A$ and any $b\in\F\setdiff A$, $a\leq c\leq b$. A cut of $\F$ that does not have a cut point is a \emph{gap}. We say that $\F$ is a \emph{complete ordered field} if $\F$ satisfies the axiom:
\begin{enumerate}
\item[\CA] {\it Cut Axiom.} Every cut of $\F$ is not a gap.
\end{enumerate} 
Otherwise, $\F$ is said to be \emph{incomplete}. 

Let $a\in\F$. If $b\in\F$ such that $a\leq b$, then we define $\max\{a,b\}:=b$, $\min\{a,b\}:=a$, and $|a|:=\max\{-a,a\}$.

\section{Equivalence of Littlewood's Principles to the Cut Axiom}

We now prove the equivalence of Littlewood's three principles to \CA, with a couple more statements added to our list of equivalences.

\begin{theorem} For an arbitrary ordered field $\F$, each of the statements
\begin{enumerate}\item[\LIP] \emph{Lebesgue Integral Property.}\footnote{The names \LIP\  and \LMP\  are not popularly used in standard real analysis. Also, these two statements reduce to trivialities when $\F$ is complete. We chose to name these two statements as such, in analogy to what in \cite{dev14} was called the \emph{Darboux Integral Property}, which states that every Darboux integrable function $f:\abClosed\into\F$ has a Darboux integral. According to \cite[p. 271]{dev14}, the Darboux Integral Property (in conjunction with some other statement) is one of the equivalent forms of \CA.} Given $\abClosed\sub\F$, every Lebesgue measurable function $\abClosed\into\F$ has a Lebesgue integral.
\item[\LMP] \emph{Lebesgue Measure Property.} Given $\abClosed\sub\F$, every measurable subset of $\abClosed$ has Lebesgue measure.
\item[\LPi] \emph{Littlewood's First Principle.} Given $\abClosed\sub\F$, for each measurable set $E\sub\abClosed$ and each $\varepsilon\in\P$, there exist intervals $I_1,I_2,\ldots, I_m\sub\abClosed$ such that if $U=\displaystyle\bigcup_{i=1}^mI_i$, then $\Int_a^b\chi_{E\symdiff U}<\varepsilon$.
\item[\LPii] \emph{Littlewood's Second Principle.} Given $\abClosed\sub\F$, for each Lebesgue measurable function $f:\abClosed\into\F$ and each $\varepsilon\in\P$, there exists a measurable set $E\sub\abClosed$ such that $f:\abClosed\setdiff E\into\F$ is a continuous function and that $\Int_a^b\chi_{E}<\varepsilon$.
\item[\LPiii] \emph{Littlewood's Third Principle.} Given $\abClosed\sub\F$, for each sequence $\lpar \varphi_n\rpar$ of Lebesgue measurable functions $\abClosed\into\F$ that converge to\linebreak $f:\abClosed\into\F$ almost everywhere in $\abClosed$, and each $\varepsilon\in\P$, there exists a measurable set $E\sub\abClosed$ such that the functions $\varphi_n:\abClosed\setdiff E\into\F$ converge uniformly to $f:\abClosed\setdiff E\into\F$, and that $\Int_a^b\chi_{E}<\varepsilon$.
\end{enumerate}
is equivalent to \CA.
\end{theorem}
\begin{proof} We shall prove
\begin{eqnarray}
\mCA \implies \mLIP \implies \mLMP \implies \CA,\label{TheCycle1}\\
\mCA \implies \mLPiii \implies \mLPii \implies \mLPi \implies \CA.\label{TheCycle2}
\end{eqnarray}
Let \Star\  be one of the statements \LIP, \LMP, \LPi, \LPii, or \LPiii. From standard real analysis, \Star\  is true for the ordered field $\R$. Suppose $\F$ is an ordered field that satisfies $\neg\mStar$. Then $\F$ cannot be $\R$, but we have a well-known fact\footnote{See, for instance, \cite[pp. 601--605]{spi08}.} that any complete ordered field is isomorphic to $\R$, so $\F$ is incomplete, and thus, $\neg\mCA$ holds in $\F$. We have thus proven\linebreak $\mCA\implies\mStar$ by contraposition. In particular,
\begin{eqnarray}
\mCA &\implies& \mLIP,\nonumber\\
\mCA & \implies & \mLPiii.\nonumber
\end{eqnarray}
The implication
\begin{eqnarray}
\mLIP &\implies& \mLMP,\nonumber
\end{eqnarray}
is trivial, while \LPii\  is well-known as a consequence of \LPiii. One proof of
\begin{eqnarray}
\mLPiii\implies\mLPii,\nonumber
\end{eqnarray}
can be found in \cite[p. 110]{cha95}, and in this proof, there is a straightforward use of \LPiii\  to carry out a `modus ponens' argument to prove \LPii, and all notions [mainly, continuity and uniform convergence] used in the proof are valid for the arbitrary ordered field $\F$, where such necessary notions have been defined for $\F$ in Section~\ref{PrelSec}. To complete the proof of \eqref{TheCycle1}--\eqref{TheCycle2}, only three implications remain, which we prove in the following.
\newline

\noindent $\neg\mCA\implies \neg\mLMP$. Suppose $\F$ is incomplete. We proceed by contradiction, so suppose \LMP\  holds. By \cite[Lemma~B, p. 110]{tei13}, there exist a gap $A\sub\F$ and some strictly increasing and non-convergent sequence $\aseq$ in $A$ such that
\begin{eqnarray}
\forall x\in A\quad \lbrak\  a_1<x \  \implies\   \exists!n\in\N\   x\in\lpar a_n,a_{n+1}\rbrak\  \rbrak.\label{TeiSeq}
\end{eqnarray} 
Let $a:=a_1$, let $b\in\F\setdiff A$, and let $E:=\abClosed\setdiff A$. For each $n\in\N$, define $\varphi_n:\abClosed\into\F$ by $\varphi_n:x\mapsto 1-\chi_{\lbrak a,a_{n}\rbrak}(x)$. Since, for any $n\in\N$,\linebreak $\lbrak a,a_n\rbrak\sub\lbrak a,a_{n+1}\rbrak$, we find that $\lpar\varphi_n\rpar$ is a monotonically decreasing sequence of step functions, with $\Int_a^b\varphi_n=b-a_{n}$ for any $n$. Let $x\in\abClosed$, and let $\varepsilon\in\P$. If $x\in A$, then $\chi_E(x)=0$, and by \eqref{TeiSeq}, there exists $N\in\N$ such that for any $n\geq N+1$, $\varphi_n(x)=0$, and furthermore, $\chi_E(x)-\varphi_n(x)=0\in\lpar-\varepsilon,\varepsilon\rpar$. If $x\in\abClosed\setdiff A$, since $\aseq$ is a sequence in $A$, by the definition of cut, we have, for any  $n\in N$, $x>a_{n}$, so $x\notin\lbrak a,a_{n}\rbrak$. This implies $\varphi_n(x)=1$, but $\chi_E(x)=1$, so we further obtain $\chi_E(x)-\varphi_n(x)=0\in\lpar-\varepsilon,\varepsilon\rpar$. We have thus shown that the statement $\chi_E(x)=\nseqLim\varphi_n(x)$ is true for any $x\in\abClosed$, and consequently, almost everywhere in $\abClosed$. This means that $E$ is measurable, and by the \LMP, there exists $\Int_a^b\chi_E\in\F$ such that $\Int_a^b\chi_E=\nseqLim\Int_a^b\varphi_n=\nseqLim(b-a_n)$. That is, the sequence $\lpar b-a_n\rpar$ converges to $\Int_a^b\chi_E$. Routine arguments may be used, to show that, in $\F$, the usual linearity and constant rules for sequence limits hold, so  $(a_n-b)$ converges to $-\Int_a^b\chi_E$ and taking the sum of $(a_n-b)$ with the constant sequence $(b)$, we find that $\aseq$ converges to $b-\Int_a^b\chi_E$, contradicting the fact that $\aseq$ is non-convergent. Therefore, $E$ does not have Lebesgue measure.\newline

\noindent $\mLPii\implies\mLPi$. Let $\varepsilon\in\P$. If $E\sub\abClosed$ is a measurable set, then $\chi_E$ is a measurable function, and by \LPii, there exists a measurable set $F\sub\abClosed$ such that $\chi_E:\abClosed\setdiff F\into\P\cup\{0\}$ is a continuous function and that $\Int_a^b\chi_{F}<\frac{\varepsilon}{2}$. Since the integral $\Int_a^b\chi_{F}$ of the nonnegative characteristic function $\chi_F$, is nonnegative, 
\begin{eqnarray}
-\frac{\varepsilon}{2}<\Int_a^b\chi_{F}<\frac{\varepsilon}{2},\label{xLebInt}
\end{eqnarray}
which, in particular, means that $F$ has Lebesgue measure. Since $ F$ is measurable, there exists a monotonically decreasing sequence $\lpar\varphi_n\rpar$ of step functions $\abClosed\into\P\cup\{0\}$ that converge to $\chi_{ F}$ almost everywhere in $\abClosed$, and that
\begin{eqnarray}
\Int_a^b\chi_{ F} = \nseqLim\Int_a^b\varphi_n.\nonumber
\end{eqnarray}
Consequently, there exists $N\in\N$ such that for all $n\geq N$, 
\begin{eqnarray}
-\frac{\varepsilon}{2}< \Int_a^b\varphi_n-\Int_a^b\chi_{ F}<\frac{\varepsilon}{2},\nonumber
\end{eqnarray}
which, in conjunction with \eqref{xLebInt}, completes the proof that
\begin{eqnarray}
\nseqLim\Int_a^b\varphi_n=0.\label{xLebInt2}
\end{eqnarray} 

For each $m\in\N$, define $P_m=\{x\in\abClosed\  :\  \chi_{ F}\geq\frac{1}{m}\}$. Since $\chi_{ F}$ takes on only values of $0$ or $1$,
\begin{eqnarray}
\bigcup_{m\in\N} P_m =\{x\in\abClosed\  :\  \chi_{ F}(x)=1\}= F.\label{xLebIntII2}
\end{eqnarray}
Since $\lpar\varphi_n\rpar$ is monotonically decreasing and converges to $\chi_F$ at $x$, by a routine argument, for any $n\in\N$, $\varphi_n\geq \chi_{ F}$, and so, 
\begin{eqnarray}
\varphi_n(x)\geq\frac{1}{m},\label{xLebIntII1}
\end{eqnarray}
for all $x\in P_m$. Since $\varphi_n$ is a step function, if we take the values of $\varphi_n$ that are at least $\frac{1}{m}$, then we have a finite subset of $\P\cup\{0\}$, the inverse image of which, under $\varphi_n$, is the union of a finite number of intervals, and this union contains $P_m$. Let such intervals be collected in the set $\mathscr{C}_n$, and let $\Sigma_n$ be the sum of the lengths of the intervals in $\mathscr{C}_n$. By \eqref{xLebIntII1}, $\Int_a^b\varphi_n\geq \frac{\Sigma_n}{m}$, and by \eqref{xLebInt2}, given $\varepsilon\in\P$, there exists $N\in\N$ such that for all $n\geq N$,
$\frac{\displaystyle\Sigma_n}{m}\leq \Int_a^b\varphi_n<\frac{\varepsilon}{m}$, and so, $\Sigma_n<\varepsilon$. At this point, we have proven that, for any $m\in\N$, $P_m$ has measure zero, and by \eqref{xLebIntII2}, so does $ F$.

Let $\lpar \phi_n\rpar$ be a monotonically decreasing sequence of step functions\linebreak $\abClosed\into\P\cup\{0\}$ such that $\chi_E(x)=\nseqLim\phi_n(x)$ for all $x\in\abClosed\setdiff Q$, where $Q$ has measure zero. For each $n\in\N$, define $\psi_n:\abClosed\into\P\cup\{0\}$ by 
\begin{eqnarray}
\psi_n(x):=\begin{cases} 1, & \mbox{if }\phi_n(x)\geq \frac{1}{2},\\
0, & \mbox{if }\phi_n(x)<\frac{1}{2},\end{cases}\nonumber
\end{eqnarray}
which is a step function. By a routine epsilon argument, 
\begin{eqnarray}
\nseqLim\left|\psi_n(x)-\chi_E(x)\right|=0,\label{Lusin1}
\end{eqnarray}
for any $x\in\abClosed\setdiff Q$. By an argument similar to that done earlier in this proof, we have $\psi_n=\chi_{B_n}$, where $B_n$ is the union of a finite number of intervals in $\abClosed$. By a routine argument, $\left|\psi_n-\chi_E\right|=\chi_{E\symdiff B_n}$. Let $D_n$ be the set of all discontinuities of $\psi_n$. Since $\psi_n$ is a step function, and $\chi_E:\abClosed\setdiff F\into\P\cup\{0\}$ is a continuous function, the set $D_n\cup F$ of all discontinuities of  $\chi_{E\symdiff B_n}=\left|\psi_n-\chi_E\right|$ is a set of measure zero, and so is $R:=Q\cup F\cup\displaystyle\bigcup_{n\in\N}D_n$. 

Let $x\in\abClosed\setdiff R$. Thus, for each $n\in\N$, $\chi_{E\symdiff B_n}$ is continuous at $x$, and since $\chi_{E\symdiff B_n}$ takes on a value of only $0$ or $1$, we find that there exists an interval $Z_x\sub\abClosed$ that contains $x$ and that $\chi_{E\symdiff B_n}$ is zero on all of $Z_x$, or there exists an interval $O_x\sub\abClosed$ that contains $x$ and that $\chi_{E\symdiff B_n}$ has value $1$ on all of $O_x$. Let $\varepsilon_0=\min\left\{1,\varepsilon\right\}$. By \eqref{Lusin1}, there exists $N\in\N$ such that for all $n\geq N$,
\begin{eqnarray}
-1 < & \chi_{E\symdiff B_n}(x) & < 1,\nonumber
\end{eqnarray}
and since $\chi_{E\symdiff B_n}(x)$ can be only $0$ or $1$, we only have $\chi_{E\symdiff B_n}(x)=0$, so $O_x=\emptyset$. Thus, $x\in Z_x$, and we have proven
\begin{eqnarray}
\abClosed\setdiff R\sub\bigcup_{x\in\abClosed\setdiff R}Z_x\sub\{x\in\abClosed\  :\  \chi_{E\symdiff B_n}=0\},\nonumber
\end{eqnarray}
from which we deduce that $E\symdiff B_n=\{x\in\abClosed\  :\  \chi_{E\symdiff B_n}\neq 0\}\sub R$.

At this point, we have proven that for all $n\geq N$, $\chi_{E\symdiff B_n}$ has measure zero, and hence has Lebesgue measure. In particular, $\Int_a^b\chi_{E\symdiff B_n}=0$. Thus, we may state that, for each $\varepsilon\in\P$, there exists a union $U=B_N$ of a finite number of intervals such that $ \Int_a^b\chi_{E\symdiff U}=0<\varepsilon$. Therefore, $\F$ satisfies \LPi.
\newline

\noindent $\neg\mCA\implies\neg\mLPi$. Suppose $\F$ is incomplete, and that, tending towards a contradiction, \LPi\  holds. As shown in the proof of $\neg\mCA\implies \neg\mLMP$, there exist a gap $A\sub\F$, some $a\in A$ and $b\in F\setdiff A$ such that $E:=\abClosed\setdiff A$ is a measurable set, but does not have Lebesgue measure. Let $\varepsilon\in\P$. By \LPi, there exist intervals $I_1,I_2,\ldots,I_\mu\sub\abClosed$ such that if  $U=\displaystyle\bigcup_{k=1}^\mu I_k$, then $\Int_a^b\chi_{E\symdiff U}<\frac{\varepsilon}{2}$. By an argument similar to one of those that were done in the proof of $\mLPii\implies\mLPi$, $E\symdiff U$ has measure zero.

Without loss of generality, we assume $I_1,I_2,\ldots,I_\mu$ are pairwise disjoint, and that, for some index $K\in\{1,2,\ldots,\mu\}$, the interval $I_K$, if nonempty, intersects both $\abClosed\cap A$ and $\abClosed\setdiff A$, and we further assume that all intervals $I_k$ with $k<K$ are subsets of $\abClosed\setdiff A$, while all intervals $I_k$ with $k>K$ are subsets of $\abClosed\cap A$. 

Let $T:=\displaystyle\bigcup_{k=1}^KI_k$. Since $E\symdiff U$ has measure zero, the sets $I_K\setdiff E\sub E\symdiff U$ and $E\symdiff T\sub E\symdiff U$ also have measure zero. Furthermore, the equation 
\begin{eqnarray}
\chi_E(x)=\chi_T(x)+\chi_{E\symdiff T}(x),\label{LebInt4}
\end{eqnarray}
is true for all $x\in\abClosed$ except those $x\in I_K\setdiff E$. Hence, \eqref{LebInt4} is true almost everywhere in $\abClosed$. Since $E\symdiff T$ has measure zero, it has Lebesgue measure, so $\chi_{E\symdiff T}$ has a Lebesgue integral. The other function in the right-hand side of \eqref{LebInt4}, which is $\chi_T$, is a step function because $T$ is an interval, and thus, $\chi_T$ also has a Lebesgue integral, and, since \eqref{LebInt4} is true almost everywhere in $\abClosed$, by a routine argument, we find that $\chi_E$ also has a Lebesgue integral, contradicting the fact that $E$ does not have Lebesgue measure. Therefore, \LPi\  is false in $\F$.
\end{proof}


\end{document}